\documentclass[11pt]{article}
\usepackage{amsmath, amsfonts, amssymb}
\usepackage{mathrsfs,dsfont}
\usepackage{graphicx, amsthm,color}
\usepackage{enumitem}

\providecommand{\U}[1]{\protect\rule{.1in}{.1in}}

\newtheorem{theorem}{Theorem}[section]
\newtheorem{lemma}[theorem]{Lemma}

\newtheorem{proposition}[theorem]{Proposition}
\newtheorem{definition}[theorem]{Definition}
\theoremstyle{definition}

\newtheorem{example}[theorem]{Example}
\theoremstyle{remark}
\newtheorem{remark}[theorem]{Remark}
\numberwithin{equation}{section}

\usepackage{tikz}
\usepackage{tikz-3dplot}
\usetikzlibrary{calc,patterns,angles,quotes}
\usetikzlibrary{babel}
\usepackage{pgfplots}
\usepgfplotslibrary{patchplots}
\usetikzlibrary{backgrounds, intersections}
\usepgfplotslibrary{fillbetween}
\usetikzlibrary{arrows,automata}
\usepackage{subcaption}

\pgfplotsset{compat=1.17}
\usepgfplotslibrary{colorbrewer}

\providecommand{\gph}{\mathop{\rm gph}\nolimits}
\providecommand{\dom}{\mathop{\rm dom}\nolimits}
\providecommand{\into}{\mathop{\rm int}\nolimits}
\providecommand{\bd}{\mathop{\rm bd}\nolimits}   
\providecommand{\argmin}{\mathop{\rm argmin}\nolimits}
\providecommand{\argmax}{\mathop{\rm argmax}\nolimits}
 
\providecommand{\tto}{\mathop{\rightrightarrows}\nolimits} 

\newcommand{\R}{\mathbb{R}}

\newcommand{\N}{\mathbb{N}}

\newcommand{\Ball}{ \mathbb{B}}
\newcommand{\Sph}{\mathbb{S}}

\providecommand{\Liminf}{\mathop{\rm Liminf}\limits}

\newcommand{\ind}{\mathds{1}}

\newcommand{\proj}{\mathrm{proj}}
\newcommand{\Proj}{\mathrm{Proj}}

\newcommand{\slope}[1]{|\nabla #1|}
\newcommand{\limSlope}[1]{\overline{|\nabla #1|}}

\definecolor{forestgreenweb}{rgb}{0.13, 0.55, 0.13}

\begin{document}
	
\begin{center}
{\LARGE Steepest geometric descent for regularized quasiconvex functions}

\vspace{1cm}
		
{\Large \textsc{Aris Daniilidis, David Salas}}
\end{center}
	
\bigskip

\noindent\textbf{Abstract.} We establish existence of steepest descent curves emanating from almost every point of a regular locally Lipschitz quasiconvex functions, where regularity means 
that the sweeping process flow induced by the sublevel sets is reversible. We then use max-convolution to regularize general quasiconvex functions and obtain a result of the same nature in a more general setting.\bigskip
	
	
\vspace{0.6cm}
	
\textit{Secondary} 34A60, 49J52, 49J53

	

\section{Introduction}\label{sec:intro}

Steepest descent curves are at the core of the theory of variational analysis, differential equations and optimization. 
Given a $\mathcal{C}^1$-smooth function $f:\R^d\to \R$ we call steepest descent curve the solution of the gradient flow equation
\begin{equation}\label{eq:Steepest-smooth}
    \begin{cases}
	\dot{x}(t) = -\nabla f(x(t)),\quad\text{a.e. }t\in[0,T],\\
	x(0) = x_{0}.
    \end{cases}
\end{equation}
It is well-known that the above differential equation has a solution (as a direct application of the Picard-Lindel\"{o}f theorem). When the assumption of smoothness is missing, the existence of steepest descent curves can still be established for convex functions. Indeed, if $f:\R^d\to \R$ is convex, its steepest descent curves are solutions of the subdifferential inclusion
	\begin{equation}\label{eq:Steepest-Convex}
		\begin{cases}
			\dot{x}(t) \in -\partial f(x(t)),\quad\text{a.e. }t\in[0,T],\\
			x(0) = x_{0}.
		\end{cases}
	\end{equation}
It is well-known that the above differential inclusion admits a (unique) solution (see, e.g., \cite{AttouchButtazzo2014}). Similarly, existence of solutions for several gradient descent and proximal methods are often based on convexity (see, e.g., \cite{PeypouquetSorin2010}).
	\medskip 
	
In the setting of metric analysis, a steepest descent curve is a $1$-Lipschitz curve verifying the metric equation
	\begin{equation}\label{eq:Steepest-Slope}
		\begin{cases}
			(f\circ x)'(t) = -\slope{f}(x(t)),\quad \text{a.e. }t\in [0,T],\\
			x(0) = x_{0},
		\end{cases}
	\end{equation}
where $\slope{f}$ denotes the (metric) slope of $f$  introduced in \cite{DeGiorgiMarinoTosques1980}. This last formulation coincides with~\eqref{eq:Steepest-smooth} in the smooth case and with~\eqref{eq:Steepest-Convex} in the convex case after performing the usual arc-length reparametrization. The metric gradient flow given by~\eqref{eq:Steepest-Slope} has been studied in detail (we refer to~\cite{AmbrosioGigliSavare2008} for a comprehensive exposition). 
Remarkable families of functions also admit steepest descent curves in the above cases:  geodesically convex functions in metric spaces \cite{AmbrosioGigliSavare2008} and smooth functions on Riemannian manifolds (see, e.g.,~\cite{Villani2009}).
 However, existence of steepest descent curves is in general hard to verify, even for Lipschitz functions in $\R^d$ (see comments of~\cite[Section~9.3.5]{Ioffe2017}).  Due to this obstruction, the authors in \cite{DDL2015, LMV2015} consider the more general notion of trajectories of a convex foliation (terminology introduced in \cite{DLS2010}) and establish existence of such orbits (see, e.g., \cite[Theorem~2.6]{DDL2015}). In case the foliation is given by the sublevel set of a quasiconvex function, the above orbits are call orbits of geometric descent. Their connection with steepest descent orbits has been explored in \cite{DrusvyatskiyIoffeLewis2015,Ioffe2017}: these curves fail to be steepest descent curves in general, but instead correspond to what the authors in \cite{DrusvyatskiyIoffeLewis2015} called \emph{curves of near-steepest descent}. One of the main difficulties is in the fact that the slope mapping $x\mapsto\slope{f}(x)$ fails to be lower-semicontinuous, inducing a gap with respect to its closure $x\mapsto\limSlope{f}(x)$. Curves of near-steepest descent lie, in some sense, within this gap. \medskip

In this work, we are interested in steepest descent curves for the class of locally Lipschitz quasiconvex functions. This  is another very important family in the context of optimization with amenable properties (see, e.g., \cite{Aussel2014}). Even though the desired existence result seems to fail for this class as well (it is not yet clear if this is the case or not), we have been able to provide a positive existence result for the class of \textit{regular} quasiconvex functions, where regularity ensures that the sweeping process flow induced by the sublevel sets is reversible. Then, for the general case of locally Lipschitz quasiconvex functions, we consider a regularization scheme using the max-convolution operator (see, e.g., \cite{SeegerVolle1995} and the references therein).  Indeed, we define for any  locally Lipschitz quasiconvex function $f:\R^d\to \R$ and $\varepsilon>0$, a regularization function $f_{\varepsilon}$ satisfying the following properties:
	\begin{enumerate}
		\item[(i).] $f_{\varepsilon}$ admits steepest descent curves for almost every initial data on its domain.
		\item[(ii).] Every critical point of $f_{\varepsilon}$ is at distance at most $\varepsilon$ of a critical point of $f$.
	\end{enumerate}

Our work borrows heavily from the geometric approach of \cite{DDL2015, DrusvyatskiyIoffeLewis2015}. We look at curves of geometric descent. Then, under regularity assumptions we are able to reverting the (unilateral) sweeping process and deduce that almost every curve of geometric descent is in fact a steepest descent curve.
	\medskip
	
	The rest of the paper is organized as follows: In Section~\ref{sec:pre} we fix our terminology and quote some preliminary results. In Section~\ref{sec:Reversible}, under adequate assumptions of the quasiconvex function (ensuring the reversibility of its sweeping process flow), we directly relate steepest descent curves as solutions of the sweeping process induced by the sublevel sets. In Section~\ref{sec:main}, we extend the results of the preceding section to general quasiconvex function, by means of regularization and localization. 
	
\section{Preliminaries}\label{sec:pre}

Throughout this work, we consider the Euclidean space $\R^d$ endowed with its usual inner product $\langle \cdot,\cdot\rangle$ and its induced norm $\|\cdot\|$. We denote by $B(x,r)$  (respectively, $\bar B(x,r)$) the open (respectively, closed) ball centered at $x$ of radius $r>0$ and by $\Ball_d$ (respectively, $\Sph_d$) the unit closed ball (respectively, the unit sphere). For a set $A\subset \R^d$, we denote by $\into(A)$, $\overline{A}$, $\bd(A)$ and $A^{\circ}$ its interior, closure, boundary and (negative) polar set, respectively.
	\medskip
	
	For a function $f:\R^d\to \R\cup\{+\infty\}$ and $\alpha\in\R$, we denote by $[f\leq \alpha]$, the $\alpha$-sublevel set of $f$, that is, 
	\begin{equation}\label{eq:SublevelSet}
		[f\leq \alpha] = \{x \in\R^d\mid f(x)\leq \alpha \}.
	\end{equation}
	Similarly, we define the strict $\alpha$-sublevel set $[f<\alpha]$, and the corresponding sets $[f=\alpha]$, $[f>\alpha]$ and $[f\geq \alpha]$. We denote its (effective) domain by $\dom f$, that is, $\dom f = \{x \in \R^d\mid f(x) <+\infty\}$. 
A function $f:\R^d\to \R\cup\{+\infty\}$ is called coercive if for every $\alpha<\sup f$ the sublevel set $[f \leq \alpha]$ is compact.
	\medskip
	
	A function $f:\R^d\to \R\cup\{+\infty\}$ is said to be \emph{quasiconvex} if
	\begin{equation}\label{eq:quasiconvex}
		\forall x,y\in\R^d,\,\forall t\in[0,1],\quad f(tx + (1-t)y)\leq \max\{ f(x),f(y)\}.
	\end{equation}
	It is well-known that $f$ is quasiconvex if and only if every sublevel set $[f\leq \alpha]$ is convex. Recall that a function is lower semicontinuous (lsc, for short) if the sublevel sets are closed.

	For a function $f:\R^d\to \R\cup\{+\infty\}$ we define the metric slope $\slope{f}$ as
	\begin{equation}\label{eq:slope-def}
		\slope{f}(x) := \left\{\begin{array}{cl}
			\displaystyle\limsup_{y\to x} \frac{(f(x)-f(y))^+}{\|x-y\|},  & \text{ if }x\in \dom f \\
			+\infty,  & \text{otherwise, } 
		\end{array}\right.
	\end{equation}
	where $a^+ = \max\{a,0\}$. The metric slope enjoys several interesting properties (see, e.g., \cite{AzeCorvellec2004,AmbrosioGigliSavare2008}), but it is well-known that it might fail to be lower semicontinuous (see, e.g., \cite{DrusvyatskiyIoffeLewis2015}). Thus, we consider the limiting slope $\limSlope{f}$ as the lower semicontinuous closure of $\slope{f}$, that is,
	\begin{equation}\label{eq:limSlope-def}
		\limSlope{f}(x) = \liminf_{y\to_f x} \slope{f}(y),
	\end{equation}
	where $y\to_f x$ means $(y,f(y))\to (x,f(x))$. A point $x\in\R^d$ is called \textit{critical} for $f$ if $\limSlope{f}(x) = 0$.
	\medskip
	
	In what follows, we denote by either $d_S(x)$ or $d(x,S)$ the distance of $x\in\R^d$ to the set $S\subset \R^d$ and by $\Proj_S(x)$ or $\Proj(x;S)$ the set of nearest points from $x$ in $S$. Whenever this set is a singleton, we call this unique nearest point as \textit{metric projection} and we denote it by $\proj_S(x)$. 
	\medskip
	
	A set $S$ is said to be \emph{prox-regular} if there exists a continuous function $\rho:S\to (0,+\infty]$ such that the enlargement of $S$ given by
	\begin{equation}\label{eq:ProxReg-Enlargement}
	  U_{\rho(\cdot)}(S) = \{u\in \R^d\ : \ \exists y\in \Proj_S(u)\text{ with }d_S(u) <\rho(y)\}  
	\end{equation}
is open and the projection $\proj_S$ is well-defined on $U_{\rho(\cdot)}(S)$  (equivalently, $d_S^2$ is $\mathcal{C}^1$ on $U_{\rho(\cdot)}(S)$, see, e.g, \cite[Prop.~4 and Prop.~11]{ColomboThibault2010prox}). For $r>0$, we say that $S$ is \emph{$r$-prox-regular} if the function $\rho(\cdot)$ can be taken as $\rho\equiv r$. Every convex set is $(+\infty)$-prox-regular.
	\medskip
	
	It is well-known (see, e.g., \cite{Thibault2023}) that for a prox-regular set, Bouligand and Clarke tangent cones coincide at every point (this is known as tangential regularity) and the same applies to the classical notions of normal cones (proximal, Fr\'{e}chet, limiting, Clarke). Since we are going to work only with convex and prox-regular sets, the notions of tangent and normal cones are unambiguously defined, that is, if $S\subset \R^d$ is prox-regular and $x\in S$, we define the \emph{(Clarke) tangent cone} and the \emph{(Clarke) normal cone} of~$S$ at~$x$ by the formulae
	\begin{equation}
		T(S;x) := \Liminf_{S\ni y\to x; t\downarrow 0} \frac{1}{t}(S-y)\qquad\text{ and }\qquad N(S;x) := [T(S;x)]^{\circ},
	\end{equation}
	where $\Liminf$ is the inferior limit of sets in the sense of Painlevé-Kuratowski (see, e.g.,\, \cite{Thibault2023}).
	\medskip
	
	A set-valued map $M:A\tto B$ is a mapping that assigns to each $a\in A$ a subset $M(a)$ of~$B$. We denote the domain of $M$ and the graph of $M$ as the sets $\dom M = \{ a\in A\ :\ M(a)\neq\emptyset\}$ and $\gph M = \{ (a,b)\ :\ b\in M(a)\}$. In the particular case when $A= [0,T]$ and $B=\R^d$, we say that the set-valued map is a \emph{moving set map} (also called \textit{sweeping process map}, see \cite{DD2017}).
	\medskip
	
	For two sets $A,B\subset \R^d$, the Hausdorff distance between $A$ and $B$ is given by
	\begin{equation}\label{eq:HausdorffDistance-def}
		d_H(A,B) = \max\left\{\sup_{a\in A} d_B(a),\, \,\sup_{b\in B} d_A(b)\right\}\, \in\, [0,+\infty].
	\end{equation}
	A set-valued map $M:A\subset\R^p\tto \R^q$ is said to be Lipschitz-continuous if there exists $L>0$ such that
	\begin{equation}
		\forall x,y\in A,\quad d_H(M(x),M(y))\leq L\|x-y\|.
	\end{equation}
	\medskip
	
	Let $K:[0,T]\tto \R^d$ be a moving set with prox-regular values. We define the sweeping process differential inclusion of $K$ as
	\begin{equation}\label{eq:SweepingProcess}
		\begin{cases}
			\phantom{a}\dot{u}(t) \in - N(K(t);u(t)),\qquad\text{ a.e. }t\in [0,T]\\
			\phantom{a} u(0) = x_0\in K(0),
		\end{cases}
	\end{equation}

	
It is well known that if $K$ is Lipschitz continuous and uniformly $r$-prox-regular for some $r>0$ (that is, $K(t)$ is $r$-prox-regular for every $t\in [0,T]$), the sweeping process admits a unique solution for every initial condition $x_0\in K(0)$ (see, e.g., \cite{Thibault2003}). The following proposition surveys the main properties of such a solution in the case of the sublevel moving set $K:[0,T]\tto\R^d$ given by $K(t) = [f\leq f(x_0)-t]$. The first statement is a classical fact in the sweeping process theory (see, e.g., \cite{Thibault2003}). The second and third statements follow directly from the fact that the process $K:[0,T]\tto\R^d$ is parametrized with respect to the values of the function $f$, while the fourth statement follows from  \cite[Theorem~3.4 and Claim 3.6]{DrusvyatskiyIoffeLewis2015}. 
	\medskip
	
	
\begin{proposition} \label{prop:UniqueSolutionSweeping} 
Let $f:\R^d\to\R$ be a locally Lipschitz function and let $\alpha\in \R$ and $T>0$ be such that $\alpha-T>\inf f$. Let $K:[0,T]\tto \R^d$ be the sublevel moving set starting from $\alpha$, that is,
		\[
		K(t) = [f\leq \alpha - t],\qquad \forall t\in [0,T].
		\]
		If $K$ is Lipschitz-continuous and uniformly $r$-prox-regular, then for every $x_0\in K(0)\setminus\into(K(T))$, the sweeping process \eqref{eq:SweepingProcess} has a unique solution $u:[0,T]\to \R^d$, satisfying that
		\begin{enumerate}
			\item[(i)] $u(\cdot)$ is Lipschitz-continuous on $[0,T]$.
			\item[(ii)] For each $t\in [0, \alpha - f(x_0)]$, $u(t) = x_0$ and
			\[
			u(t) \in\bd K(t),\qquad \forall t\in [\alpha - f(x_0),T]. 
			\]
			\item[(iii)] $(f\circ u)'(t) = -1$ for a.e. $t\in [\alpha-f(x_0),T]$, and
			\item[(iv)] $u(\cdot)$ is a curve of near-maximal slope of $f$, in the sense that
			\[
			\frac{1}{|\nabla f|(u(t))} \,\leq\,\, \|\dot{u}(t)\|\, \,\leq \,\,\frac{1}{\overline{|\nabla f|}(u(t))},\quad\text{ for\, a.e. }\,\, t\in [\alpha-f(x_0),\,T].
			\]
			
		\end{enumerate}	
	\end{proposition}
	\medskip
Last but not least, following \cite{EvansGariepy2015Measure}, we denote by $DF(x)$ the derivative of a Lipschitz function $F:\R^d\to\R^d$ at each point of differentiability $x\in\R^d$ and by
	\begin{equation}
		JF(x) = |\det (DF(x))|,
	\end{equation}
the \textit{Jacobian} determinant of~$F$. We finally denote by $\mathcal{H}^m$ the $m$-dimensional Hausdorff measure.

\section{Reversible geometric descent for regular quasiconvex functions}\label{sec:Reversible}

This section is devoted to the study of geometrical curves of descent for a lower semicontinuous quasiconvex function $f:\R^d\to \R\cup\{+\infty\}$ which is \textit{continuous on the interior of its domain}. 
We further consider functions that satisfy the following regularity hypotheses:	
	\begin{enumerate}[label=(H\arabic{*}), ref=(H\arabic{*})]\setlength{\itemsep}{0.3cm}
\item \label{hyp:NonemptyIntArgmin} $f$ is coercive and $\mathrm{int } [f\leq \alpha] \neq \emptyset$, for every $\alpha > \inf f$.

\item\label{hyp:BoundedSlope} The slope of $f$ is bounded away from zero around every $x\in \dom f\setminus\argmin f$, that is, there exist $\delta,\ell>0$ such that
		\[
		\slope{f}(y) > \ell,\qquad \forall y\in B(x,\delta)\cap \dom f.
		\]
		\item\label{hyp:BoundedInternalCurvature} For every $\alpha\in (\inf f,\sup f)$, there exist $\eta,r>0$ such that for every $\beta\in (\alpha-\eta,\alpha+\eta)$, the set $\R^d\setminus \into([f\leq\beta])$ is $r$-prox-regular.
	\end{enumerate}
\smallskip

\begin{remark}
While~\ref{hyp:NonemptyIntArgmin} is a common assumption, \ref{hyp:BoundedSlope}--\ref{hyp:BoundedInternalCurvature} are rather strong regularity conditions. Hypothesis~\ref{hyp:BoundedSlope} is equivalent to say that the only critical points of $f$ are its minimizers. (Such functions are often called \textit{pseudoconvex}, see e.g. \cite{DH99}). Hypothesis \ref{hyp:BoundedInternalCurvature} entails the smoothness of the level sets (see Lemma \ref{lem-mnfd} below). We shall in Section~\ref{sec:main} that every locally Lipschitz quasiconvex function can be regularized in a way that~\ref{hyp:NonemptyIntArgmin}-\ref{hyp:BoundedInternalCurvature} are fulfilled.
\end{remark}

\medskip

Notice that hypothesis~\ref{hyp:BoundedSlope} together with continuity of $f$ on the interior of its domain yield that for every $x\in\into(\dom f)$ and $\delta>0$ such that $B(x,\delta)\subset \dom f$, setting $f(x)=r$ it holds:
	\begin{equation}
		\bd[f\leq r] \cap B(x,\delta) \,\,\,=\,\,\, [f=r] \cap B(x,\delta) .
	\end{equation}
Indeed, continuity of $f$ entails directly the left-to-right inclusion. If the reverse inclusion does not hold, then there would exist $z\in \R^d$ and $\varepsilon>0$ small enough such that $$B(z,\varepsilon) \subset  [f=f(x)] \cap B(x,\delta).$$ This would yield $\slope{f}(z) = 0$ contradicting~\ref{hyp:BoundedSlope}.

\medskip

Let us also mention that it is easy to construct a quasiconvex function $f:\R^d\to \R\cup\{+\infty\}$ whose sublevel sets have smooth boundaries, yet failing the reversibility hypothesis~\ref{hyp:BoundedInternalCurvature}.

\begin{example} Let $S:[0,2]\tto \R^2$ be a convex valued function defined by
\[
	S(s) := \begin{cases}
 \phantom{trilemihn}B(0,s),  &\mbox{ if } s\in [0,1).\\
\phantom{ok}\mathrm{co}\Big( B(0,s) \bigcup B\big((0,2s-1), s-1\big)\Big), & \mbox{ if } s\in [1,2],\smallskip\\	
	\end{cases}
\]
	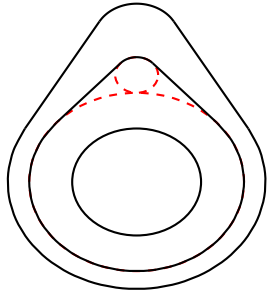
\begin{figure}[h]
		\centering
		\begin{tikzpicture}
			\begin{axis}[
				xmin=-3,xmax=3,
				ymin=-3,ymax=3,
				axis lines=none,
				]
				\draw[red,thick,dashed] (axis cs:{0},{0}) circle (5/4);
				\draw[red,thick, dashed] (axis cs:{0},{3/2}) circle (1/4);
				
				\draw[thick] (axis cs:{5/4*cos(41.8087)},{5/4*sin(41.8087)})--(axis cs:{1/4*cos(41.8087)},{3/2+1/4*sin(41.8087)});
				\draw[thick] (axis cs:{5/4*cos(180-41.8087)},{5/4*sin(180-41.8087)})--(axis cs:{1/4*cos(180-41.8087)},{3/2+1/4*sin(180-41.8087)});
				\draw[
				samples = 20,
				smooth,
				domain = 180-41.8087:360+41.8087,
				variable = \t,
				thick
				]
				plot (
				{5/4*cos(\t)},
				{5/4*sin(\t)}
				);
				\draw[
				samples = 20,
				smooth,
				domain = 41.8087:180-41.8087,
				variable = \t,
				thick
				]
				plot (
				{1/4*cos(\t)},
				{3/2 + 1/4*sin(\t)}
				);
				
				\draw[thick] (axis cs:{3/2*cos(30)},{3/2*sin(30)})--(axis cs:{1/2*cos(30)},{2+1/2*sin(30)});
				\draw[thick] (axis cs:{3/2*cos(180-30)},{3/2*sin(180-30)})--(axis cs:{1/2*cos(180-30)},{2+1/2*sin(180-30)});
				\draw[
				samples = 20,
				smooth,
				domain = 180-30:360+30,
				variable = \t,
				thick
				]
				plot (
				{3/2*cos(\t)},
				{3/2*sin(\t)}
				);
				\draw[
				samples = 20,
				smooth,
				domain = 30:180-30,
				variable = \t,
				thick
				]
				plot (
				{1/2*cos(\t)},
				{2 + 1/2*sin(\t)}
				);
				\draw[ thick] (axis cs:{0},{0}) circle (3/4);
			\end{axis} 
		\end{tikzpicture}
		\caption{{\footnotesize Boundary of $S(s)$ for $s=0,75$, $s= 1,25$ and $s=1,5$.}}
		\label{fig:Example}
	\end{figure}
	
We set $f(x)= \inf\{s : x \in S(s)\}$. Then $f$ is lower semicontinuous, quasiconvex and locally Lipschitz on the interior of its domain. Moreover, for every $s\in [0,2]$, the set $[f=s]\equiv\bd S(s)$ is a smooth manifold, whose minimal value of the internal curvature is $s-1$ for $s\in [1,2]$, and $s$ for $s\in [0,1]$.\hfill$\Diamond$
\end{example}


\bigskip	

	Let $\alpha_1,\alpha_2\in \R$ such that $\inf f <\alpha_1<\alpha_2<\sup f$. The goal of this section is to show that under~\ref{hyp:NonemptyIntArgmin}--\ref{hyp:BoundedInternalCurvature}, the function $f$ admits steepest descent curves, which locally induce a foliation of the annulus $[\alpha_1\leq f\leq \alpha_2]$.\medskip

	The first step is the following proposition that shows that prox-regularity of the boundary entails in fact smoothness of it.

	\begin{lemma}[Smoothness of the boundaries] \label{lem-mnfd} Under~\ref{hyp:NonemptyIntArgmin}--\ref{hyp:BoundedInternalCurvature}, for every $\alpha\in (\inf f,\sup f)$, the set $\mathcal{M}= \bd([f\leq\alpha])\cap \into(\dom f)$ is a $\mathcal{C}^{1,1}$-submanifold.
	\end{lemma}
	
	\begin{proof}
		Let $S = [f\leq\alpha]$ and $U = \R^d\setminus \into(S)$. Since $U$ is $r$-prox-regular, using \cite[Theorem~6.42]{RockafellarWets1998} (and noting that prox-regularity entails regularity in the sense of \cite[Definition 6.4]{RockafellarWets1998}), we get that
		\[
		T_{\bd S}(x) = T_{S}(x)\cap T_{U}(x)\quad \forall x\in \bd S.
		\]
		Moreover, since $S$ is convex with nonempty interior, we can apply \cite[Proposition 2.3]{CzarneckiThibault2018} to deduce that 
		\[
		T_{S}(x) = -T_{U}(x),\quad \forall x\in \bd S.
		\]
Combining the above equations we deduce that $T_{\bd S}(x)$ is a vector space. Using \cite[Proposition~2.113 (a5)]{Thibault2023}, we get $\R^d\setminus T_{S}(x) \subset T_U(x)$ and consequently $\bd T_S(x) \subset T_{\bd S}(x)$. We conclude that $T_{\bd S}(x)$ is of codimension 1. Therefore, for every $x\in \bd S$, there exists a unique unit vector $\hat{n}(x)\in \Sph_d$ such that
		\[
		N(S,x) = -N(U,x) = \R_+\hat{n}(x).
		\]
		Thus, $N(\bd S,x) = \R\hat{n}(x)$. Since $S$ is convex, the mapping $\hat{n}:\bd S\to \Sph_d$ is continuous. Furthermore, since $U$ is $r$-prox-regular for some $r>0$, then the set $U_{-\varepsilon} =\{ z\ :\ d_S(x)\geq \varepsilon\}$ must be $(r+\varepsilon)$-prox-regular. By noting that 
\[ 
\hat{n}(x) = \frac{1}{\varepsilon}(\proj(x,U_{-\varepsilon})-x),\quad \forall x\in \into(\dom f)\cap\bd S, 
\]
we deduce that $\hat{n}: \into(\dom f)\cap\bd S\to \Sph_d$ is also Lipschitz-continuous and the proof is complete.
	\end{proof}

\begin{remark} An alternative proof of the above lemma can be derived using the enhanced Baillon-Haddad theorem of \cite{PV2021}, by representing the convex set $[f\leq \alpha]$ as the epigraph of a convex function over an appropriate subspace of codimension~1. The above presentation aims at further describing the behavior of the tangent and normal cones of $S=[f\leq \alpha]$ and $U=\R^d\setminus\into(S)$.
	\end{remark}

	\begin{example}
Smoothness of $\bd[f\leq \alpha]$ does not entail that this set coincides with the corresponding level set $[f=\alpha]$. Discrepancies may appear due to the cutting effect of the boundary of the domain as illustrated in the following example: Set $D = \mathrm{co}\Big( \Ball_2 \cup \big( (3,0)+ \Ball_2 \big)\Big)\subset \R^2$ and define $f:\R^2\to \{+\infty\}$ given by
		
		\begin{equation}\label{eq:FunctionSmoothCutSublevelSets}
			f(x,y) = \begin{cases}
				\min\{t:\, (x,y)\in (t,0)+\Ball_2\},\quad&\text{ if }(x,y)\in D\\
				\phantom{triade}+\infty,&\text{ otherwise.}
			\end{cases}
		\end{equation}
The above function is quasiconvex and its sublevel sets are given by $[f\leq t] = \mathrm{co}(\Ball_2 \cup ((t,0)+\Ball_2))$, for $t\in [0,3]$, $[f\leq t] = \dom f = D$ if $t\geq 3$, and $[f\leq t] =\emptyset$ if $t<0$ (see Figure~\ref{fig:Example-CutSublevelSets} below). 
		
		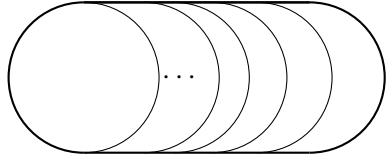
\begin{figure}[h]
			\centering
			\begin{tikzpicture}
				\draw[thick] (3,-1) arc(-90:90:1)--(0,1);
				\draw[thick] (0,-1)--(3,-1);
				\draw[thick] (0,1)--(3,1);
				\draw[thick] (0,1) arc(90:270:1)--(0,-1);
				{\foreach \x in {0.8, 1.2,1.7,2.3}
					\draw (\x,1) arc(90:-90:1)--(\x,-1);
				}
				\draw (0,1) arc(90:-90:1)--(0,-1);
				\node at (1.3,0) {$\cdots$};
			\end{tikzpicture}
\caption{{\footnotesize Each sublevel set  has a smooth boundary which does not coincide with the corresponding level set.}}
			\label{fig:Example-CutSublevelSets}
		\end{figure}
\smallskip

It is easy to see that $f$ satisfies~\ref{hyp:NonemptyIntArgmin},~\ref{hyp:BoundedSlope} and~\ref{hyp:BoundedInternalCurvature}. The second one follows from 
the remark that for every angle $\theta\in [-\pi/2,\pi/2]$ one has that
		\[ 
		f(\cos(\theta)+t,\sin(\theta)) = t,\quad \forall t\in [0,3],
		\]
and consequently $\slope{f}(x,y) \geq 1$ for all $(x,y)\in \dom f\setminus\argmin f = D\setminus\Ball_2$. The first and third hypotheses follows from the construction.\hfill$\Diamond$
	\end{example}
	\bigskip
	
In what follows, let $\alpha_1, \alpha_2 \in \R$ be such that $\alpha_1 < \alpha_2$ and set $$T:= \alpha_2-\alpha_1 >0.$$ 
We consider the \textit{annulus} set $$\mathcal{R}(\alpha_1,\alpha_2) = [\alpha_1\leq f\leq \alpha_2]$$
as well as the decreasing moving set maps $\widehat{\mathcal{S}}$, $\widehat{\mathcal{U}}$ given by
\begin{equation}\label{eq:MovingSetsSandU}
    \begin{cases}
              \phantom{tr}\widehat{\mathcal{S}}:[0,T]\tto \R^d \medskip\\
	    \phantom{tr}\widehat{\mathcal{S}}(t) := [f\leq \alpha_2 - t] 
    \end{cases}\text{ and }\quad \begin{cases}
                \phantom{tr}\widehat{\mathcal{U}}:[-T,0]\tto \R^d \medskip \\
                \phantom{tr}\widehat{\mathcal{U}}(\tau):=\underbrace{\R^d\setminus\into([f\leq \alpha_2+\tau])}_{[\into(\widehat{\mathcal{S}}(-\tau))]^c}.
    \end{cases}
\end{equation}

The following lemma establishes the variational regularity of the sublevel sets of $f$, under hypotheses~\ref{hyp:NonemptyIntArgmin}-\ref{hyp:BoundedInternalCurvature}. Namely, we prove the Lipschitz continuity of the set-valued maps given in~\eqref{eq:MovingSetsSandU}. Similar results involving metric regularity can be obtained as well using the metric slope (see, e.g., \cite{AzeCorvellec2004, Ioffe2017}). 
\begin{lemma}\label{lemma:StandardHyp} The moving set maps $\widehat{\mathcal{S}}$ and $\widehat{\mathcal{U}}$ are Lipschitz continuous provided~\ref{hyp:NonemptyIntArgmin}-\ref{hyp:BoundedSlope} hold. Moreover, $\widehat{\mathcal{U}}$ is uniformly $r$-prox-regular for some $r>0$, provided~\ref{hyp:BoundedInternalCurvature} holds.
\end{lemma}

\begin{proof}
Thanks to~\ref{hyp:BoundedSlope} for each $x\in \mathcal{R}(\alpha_1,\alpha_2)$, there exist $\delta_x>0$ and $\ell_x>0$ such that
\[
\slope{f}(z)\geq \ell_x,\qquad \text{for all }\,z\in B(x,\delta_x). 
\]
Since $\mathcal{R}(\alpha_1,\alpha_2)$ is compact due to the coercivity of $f$, we deduce that there is $\ell>0$ such that $\slope{f}(x)\geq\ell$ for all $x\in \mathcal{R}(\alpha_1,\alpha_2)$. By \cite[Theorem~2.1]{AzeCorvellec2004}, for all $x\in\mathcal{R}(\alpha_1,\alpha_2)$ and $\alpha\in [\alpha_1,\alpha_2]$ we have:
		\[
		d(x, [f\leq \alpha]) \leq \frac{1}{\ell}(f(x) - \alpha)^+.
		\]
		Now, choose $s,t\in [0,T]$ and suppose that $s<t$. Then,
		\begin{align*}
			d_H(\widehat{\mathcal{S}}(t),\widehat{\mathcal{S}}(s)) &= \sup_{x\in \widehat{\mathcal{S}}(s)} d(x,\widehat{\mathcal{S}}(t)) = \sup_{x\in \widehat{\mathcal{S}}(s)} d(x,[f \leq \alpha_1 - t]) \\
			&\leq \sup_{x\in \widehat{\mathcal{S}}(s)} \frac{1}{\ell} (f(x) - \alpha_1 + t)^+= \frac{1}{\ell}|t-s|.
		\end{align*}
		Thus, $\widehat{\mathcal{S}}$ is a $(1/\ell)$-Lipschitz set-valued map. Now take
\[
x\in \widehat{\mathcal{U}}(-t)\setminus \widehat{\mathcal{U}}(-s) = [\mathrm{int} (\widehat{\mathcal{S}}(t))]^c\setminus[\mathrm{int} (\widehat{\mathcal{S}}(s))]^c= \widehat{\mathcal{U}}(-t) \cap \into(\widehat{\mathcal{S}}(s))
\]
(notice that by~\ref{hyp:NonemptyIntArgmin} the above set is nonempty) and let $\hat{n}$ the exterior unit vector of $\widehat{\mathcal{S}}(t)$ at $\proj_{\widehat{\mathcal{S}}(t)}(x)$. Then, since $f$ is coercive, there exists $\zeta>0$ such that $\proj_{\widehat{\mathcal{S}}(t)}(x)+\zeta\hat{n}\in \bd \widehat{\mathcal{S}}(s) = \bd \widehat{\mathcal{U}}(-s)$. Clearly 
		\[
		d(x,\widehat{\mathcal{U}}(-s))\leq \zeta = d(\proj_{\widehat{\mathcal{S}}(t)}(x)+\zeta\hat{n},\widehat{\mathcal{S}}(t)) \leq \sup_{y\in \widehat{\mathcal{S}}(s)} d(y,\widehat{\mathcal{S}}(t)).
		\] 
		With this in mind, we can write
		\begin{align*}
			d_H(\widehat{\mathcal{U}}(-t),\widehat{\mathcal{U}}(-s)) = \sup_{x\in \widehat{\mathcal{U}}(-t)\,\cap\, \widehat{\mathcal{S}}(s)} d(x,\widehat{\mathcal{U}}(-s))
			&\leq \sup_{y\in \widehat{\mathcal{S}}(s)}d(y,\widehat{\mathcal{S}}(t))\leq \frac{1}{\ell}|t-s|.
		\end{align*}
Thus, $\widehat{\mathcal{U}}$ is also $\frac{1}{\ell}$-Lipschitz. The last assertion of the statement is straightforward.
\end{proof}
	
	
In what follows, we consider the sweeping process~\eqref{eq:SweepingProcess} for the moving set maps $\widehat{\mathcal{S}}:[0,T]\tto\R^d$ and $\widehat{\mathcal{U}}:[-T,0]\tto\R^d$, that is,
	\begin{equation}\label{eq:SweepingSandU}
		\begin{aligned}
			\begin{cases}
				\phantom{a}\dot{u}(t)\in -N(\widehat{\mathcal{S}}(t);u(t)),\,\,\, t\in [0,T] \smallskip\\
				\phantom{a}u(0) = x_0 \in \widehat{\mathcal{S}}(0)
			\end{cases}
			&\text{ and }\quad
			&\begin{cases}
				\phantom{a}\dot{v}(\tau)\in -N(\widehat{\mathcal{U}}(\tau);v(\tau)),\,\,\,\tau\in [-T,0] \smallskip\\
				\phantom{a}v(-T) = y_0 \in \widehat{\mathcal{U}}(-T)
			\end{cases}
		\end{aligned}
	\end{equation}
	
We now set 
\begin{equation}\label{eq:R}
M = \bd([f\leq\alpha_2]) \qquad \text{ and } \qquad \mathcal{R}= [f\leq\alpha_2]\setminus \into([f\leq\alpha_1]),
\end{equation}
and consider the mapping
\begin{align}
	u:[0,T]&\times M\to \mathcal{R} \label{eq:u}\\
	(t,m)&\mapsto u(t,m), \notag
\end{align}
where $u(\cdot,m)$ is the unique solution of the first sweeping process differential inclusion of \eqref{eq:SweepingSandU} with initial condition $x_0 = m$.  In what follows, we endow the set $[0,T]\times M$ with the distance:
\begin{equation}\label{eq:D}
D\left((t_1,m_1), (t_2,m_2) \right) := |t_1-t_2| + \|m_1-m_2\|,\quad \forall t_1,t_2\in [0,T],\, \forall m_1,m_2 \in M.
\end{equation}
\begin{proposition}[Inversion of the sweeping flow]\label{prop:InvertibleSweeping} Let $m\in M\,\bigcap\,\into(\dom f)$. Suppose that $f$ is Lipschitz-continuous around $m\in M$. Then, under~\ref{hyp:NonemptyIntArgmin}--\ref{hyp:BoundedInternalCurvature}, the mapping $u:[0,T]\times M\to \mathcal{R}$ is one-to-one and bi-Lipschitz around 
$(0,m)$.
\end{proposition}

\begin{proof}
Let us first show that $u$ is (globally) Lipschitz. Let us denote by $K$ the (common) Lipschitz constant of the moving set maps $\widehat{\mathcal{S}}(\cdot)$ and $\widehat{\mathcal{U}}(\cdot)$, given by Lemma~\ref{lemma:StandardHyp}.
Notice that every trajectory of a $K$-Lipschitz sweeping process map is a Lipschitz curve with the same constant $K$ (see e.g. \cite{Thibault2003}), that is,  for every $m\in M$ and $t_1,t_2\in [0,T]$ it holds
\[
\|u(t_1,m)-u(t_2,m)\|\leq K |t_1-t_2|.
\]
 Recall also (see \cite{Moreau1977}) that since $\widehat{\mathcal{S}}$ has convex values, the distance between two different trajectories is decreasing in time. We deduce directly that for $(t_1,m_1), (t_2,m_2)\in [0,T]\times M$ it holds:
		\begin{align*}
			\| u(t_1,m_1) - u(t_2,m_2) \| &\leq \|u(t_1,m_1) - u(t_2,m_1)\| +\|u(t_2,m_1) - u(t_2,m_2)\| \smallskip \\
			&\leq \Big( \max\{1,K\} \Big) D\left((t_1,m_1), (t_2,m_2) \right) 
		\end{align*} 
Take $\delta>0$ so that $B(m,\delta) \subset \mathrm{int} (\dom f)$. Recalling the definition of the set $M$ from~\eqref{eq:R}, since $f$ is continuous on $B(m, \delta)$, we deduce that 
$$\Gamma \, := \, M\cap B(m,\delta) = [f =\alpha_2]\cap B(m, \delta).$$ 
Moreover, by continuity of the mapping $u$ (given in~\eqref{eq:u}), we get that there exists $\varepsilon>0$ such that
		\[
u([0,\varepsilon]\times \Gamma)\subset \into(\dom f).
		\]
Shrinking $\delta$ and $\varepsilon$ if necessary, we can assume that $f$ is Lipschitz-continuous on $u([0,\varepsilon]\times \Gamma)$. Evoking Proposition~\ref{prop:UniqueSolutionSweeping} we deduce that for every $(t,m)\in (0,\varepsilon]\times \Gamma$, $f(u(t,m)) = \alpha_2-t$, that is, the mapping $u$ takes values in $\mathcal{R}$, given in~\eqref{eq:R}. Moreover,
		\begin{equation}\label{eq:L}
		|t_1 - t_2| = |f(u_1) - f(u_2)| \,\leq\, L\| u_1 - u_2 \|,
		\end{equation}
where $L$ is the Lipschitz constant of $f$ on the compact set $u([0,\varepsilon]\times \Gamma)\subset \into(\dom f)$. \smallskip\newline
Let us now show that $u(\cdot,\cdot)$ is one-to-one on $[0,\varepsilon]\times \Gamma$. \smallskip\newline
To this end, let $(t_1,m_1),(t_2,m_2)\in [0,\varepsilon]\times \Gamma$ be such that $u(t_1,m_1) = u(t_2,m_2)$. Since 
$$f(m_1) = f(m_2) = \alpha_2,$$ we get that
		\[
		\alpha_2-t_1 = f(u(t_1,m_1)) = f(u(t_2,m_2)) = \alpha_2-t_2,
		\]
which yields that $t_1=t_2$. Let us denote by $\bar{t}$ the common value of $t_1=t_2$, and by $\bar{u}$ the common value of $u(t_1,m_1) = u(t_2,m_2)$. Consider the differential inclusion
		\begin{equation}\label{InProof:InverseDiffInclusion}
			\begin{cases}
				\phantom{a}\dot{v}(t) \in - N(\widehat{\mathcal{U}}(t); v(t)),\qquad t\in [-\bar{t},0],\\
				\phantom{a}v(-\bar{t}) = \bar{u}.
			\end{cases}
		\end{equation}
Hypotheses~\ref{hyp:NonemptyIntArgmin}--\ref{hyp:BoundedInternalCurvature} ensure that $\widehat{\mathcal{U}}:[-\bar{t},0]\tto\R^d$ is uniformly prox-regular and Lipschitz continuous, entailing that the above differential inclusion has a unique solution $v:[-\bar{t},0]\to \R^d$ (\textit{c.f.}  Proposition~\ref{prop:UniqueSolutionSweeping}). Noting that both $t\mapsto u(-t,m_1)$ and $t\mapsto u(-t,m_2)$ are solutions of the above differential inclusion, we deduce that
		\[
		m_1 = u(0,m_1) = v(0) = u(0,m_2) = m_2,
		\]
and consequently $(t_1,m_1) = (t_2,m_2)$, proving that $u(\cdot,\cdot)$ is one-to-one. \smallskip\newline
Let us now show that the flow can be reversed, and that the expansion of the reversed flow can be controlled: 
to this end, let $(t_1,m_1),(t_2,m_2) \in [0,\varepsilon]\times\Gamma$, with $t_1\leq t_2$, and set $\bar{u}_i = u(t_i,m_i)$, 
for $i \in \{1,2\}$. Consider the differential inclusions, for $i\in\{1, 2\}$:
		\begin{equation*}
			\begin{cases}
				\phantom{a}\dot{v}_i(t) \in - N(v_i(t), \widehat{\mathcal{U}}(t)),\qquad t\in [-t_2,0],\\
				\phantom{a} v_i(-t_2) = \bar u_i\,.
			\end{cases}\
		\end{equation*}
Thanks to Lemma~\ref{lemma:StandardHyp} (see, e.g., \cite{Thibault2003}) and  Proposition~\ref{prop:UniqueSolutionSweeping}, the above differential inclusions have unique solutions, 
$v_1,v_2:[-t_2,0]\to \R^d$. It is not hard to see that $v_2(t) = u(-t,m_2)$ for every $t\in [-t_2,0]$ and that
\[
v_1(t) = \begin{cases}
\phantom{tri}\bar u_1,\qquad&\mbox{ if }t\in [-t_2,-t_1]\\
u(-t,m_1),&\mbox{ if }t\in [-t_1,0].
\end{cases}
\]
Let $\mathcal{P}_k := \{ s_0,\ldots,s_k \}$ be a uniform partition of the interval $[-t_2,0]$ with width
\[
|\mathcal{P}_k| := |s_{j}-s_{j-1}|=\frac{t_2}{k},\qquad\text{for all }\, j\in\{1,\dots,k\}.
\]
Let $r>0$ be such that the sets $\widehat{\mathcal{U}}(t)$ are $r$-uniformly prox-regular for all $t\in [-t_2,0]$ and take $k$ sufficiently large such that 
\begin{equation}\label{eq:theta}
\theta\,:=\,\frac{K\,|\mathcal{P}_k|}{r}=\frac{K\, t_2}{r\,k}\,<\,1.
\end{equation} 
Then for $i\in\{1,2\}$, we define the polygonal curve $v_{i,k}$ emanating from $u_i$ associated to $\mathcal{P}_k$ as follows:
		\[
		v_{i,k}(s_j) = \begin{cases}
			\phantom{trilem} \bar u_i,\qquad&\mbox{ if }j=0\\
			\phantom{a}\proj(v_{i,k}(s_{j-1}); \widehat{\mathcal{U}}(s_j)),&\mbox{ if }j\in \{1,\ldots,k\}.
		\end{cases}
		\]
Notice that all projections as well-defined. In particular, for every $j\in \{ 1,\ldots,k \}$
		\[
	v_{1,k}(s_{j-1}), v_{2,k}(s_{j-1}) \,\in\, \widehat{\mathcal{U}}(s_j) + \big( \underbrace{K|s_j - s_{j-1}|}_{\theta\,r}\big)\Ball_d \,=\, \widehat{\mathcal{U}}(s_j) + \theta\, r\,\Ball_d.
		\]
Since the projection $\proj(\cdot, \widehat{\mathcal{U}}(s_j))$ is Lipschitz with constant $\left (1 - \theta\right )^{-1}$ on the set 
\[
\widehat{\mathcal{U}}(s_j) + K|sj - s_{j-1}|\Ball_d \,\equiv \, \widehat{\mathcal{U}}(s_j) + \theta\,r\, \Ball_d 
\]
(see, e.g., \cite{ColomboThibault2010prox}) we obtain
		\begin{align*}
			\| v_{1,k}(s_j) - v_{2,k}(s_j) \| \leq (1-\theta)^{-1}\,\| v_{1,k}(s_{j-1}) - v_{2,k}(s_{j-1}) \|,\quad j\in\{1,\dots,k\}
		\end{align*}
and we deduce $\| v_{1,k}(0) - v_{2,k}(0) \| \,\leq \,(1-\theta)^{-k} \,\|u_1 - u_2\|.$\smallskip\newline
Setting $\gamma:=r^{-1}K\,t_2$ and recalling~\eqref{eq:theta}), we obtain 
\[
\| v_{1,k}(0) - v_{2,k}(0) \|  \, \leq \, \left(1- \frac {\gamma} {k}\right)^{-k} \,\|u_1 - u_2\|.
\]
Since $v_{i,k}$ converges uniformly to $v_i$ as $k\to \infty$ (see, e.g., \cite{ColomboThibault2010prox}), we conclude that
		\[
		\| m_1-m_2 \| = \| v_{1}(0) - v_{2}(0) \|\,\leq \,e^{r^{-1} K t_2}\| u_1 - u_2 \|.
		\]
Recalling~\eqref{eq:L} we conclude that
		\[
		D((t_1,m_1), (t_2,m_2)) \leq \left( L + e^{r^{-1} K T} \right)\|u(t_1,m_1) - u(t_2,m_2)\|.
		\]
This shows that $u$ is bi-Lipschitz and the proof is complete.
	\end{proof}
	
	
The above proposition shows that the mapping $u:[0,T]\times M\to\mathcal{R}$ locally induces a foliation of the annulus $\mathcal{R} = [\alpha_1\leq f\leq \alpha_2]$, near every point $m\in M\cap\into(\dom f)$ satisfying that $f$ is Lipschitz-continuous on a neighborhood of $m$. The problem appears at points in $\bd(\dom f)$ that belong to the boundary of several sublevel sets, since at these points the mapping $u$ loses injectivity. However, if $\dom f = \R^d$ and $f$ is locally Lipschitz, the neighborhood $[0,\varepsilon]\times \Gamma$ can be taken to be $[0,T]\times M$ and $u:[0,T]\times M\to \mathcal{R}$ induces a (complete) foliation of the whole annulus $\mathcal{R}$.  This is illustrated in Figure~\ref{fig:Foliation}.
		
	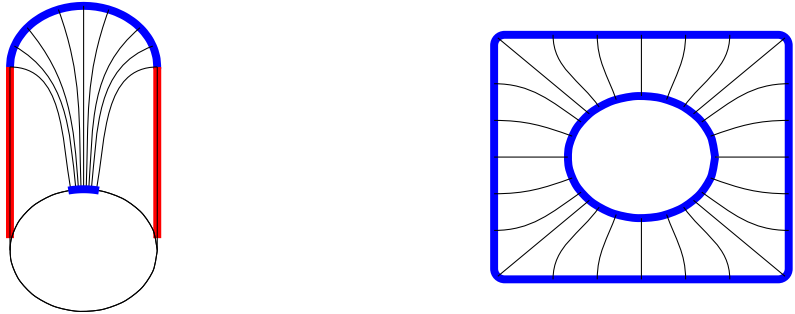
\begin{figure}[h]
		\begin{minipage}{0.45\textwidth}
		\centering
		\begin{tikzpicture}
			\begin{axis}[
				axis lines = none,
				ymin = -5, ymax = 2,
				xmin = -3.5, xmax = 3.5 
				]
				
				\draw[ samples = 20,	smooth, domain = -2:182, blue, line width=3pt,
				variable = \t	] plot ( 
				{cos(\t)}, {sin(\t)}
				);
				\draw[red, line width = 3pt] (axis cs: {-1},{0})--(axis cs: {-1},{-2.8});
				\draw[red, line width = 3pt] (axis cs: {1},{0})--(axis cs: {1},{-2.8});
				\draw[black] (axis cs: {-1},{0})--(axis cs: {-1},{-3});
				\draw[black] (axis cs: {1},{0})--(axis cs: {1},{-3});
				\draw[ samples = 20,	smooth, domain = 0:360, black,
				variable = \t	] plot ( 
				{cos(\t)}, {-3+sin(\t)}
				);
				
				\draw[ samples = 20,	smooth, domain = 0:360, black,
				variable = \t	] plot ( 
				{cos(\t)}, {-3+sin(\t)}
				);
				
				\draw[] (axis cs: {cos(80)},{-3+sin(80)}) to[out=80,in=180] (axis cs: {cos(0)},{sin(0)});
				\draw[] (axis cs: {cos(84)},{-3+sin(84)}) to[out=84,in=200] (axis cs: {cos(20)},{sin(20)});
				\draw[] (axis cs: {cos(86)},{-3+sin(86)}) to[out=86,in=220] (axis cs: {cos(40)},{sin(40)});
				\draw[] (axis cs: {cos(88)},{-3+sin(88)}) to[out=88,in=250] (axis cs: {cos(70)},{sin(70)});
				\draw[] (axis cs: {cos(90)},{-3+sin(90)}) to[out=90,in=270] (axis cs: {cos(90)},{sin(90)});
				\draw[] (axis cs: {cos(92)},{-3+sin(92)}) to[out=92,in=290] (axis cs: {cos(110)},{sin(110)});
				\draw[] (axis cs: {cos(94)},{-3+sin(94)}) to[out=94,in=310] (axis cs: {cos(140)},{sin(140)});
				\draw[] (axis cs: {cos(96)},{-3+sin(96)}) to[out=96,in=330] (axis cs: {cos(160)},{sin(160)});
				\draw[] (axis cs: {cos(100)},{-3+sin(100)}) to[out=100,in=0] (axis cs: {cos(180)},{sin(180)});

				\draw[ samples = 20,	smooth, domain = 78:102, blue, line width=3pt,
				variable = \t	] plot ( 
				{cos(\t)}, {-3+sin(\t)}
				);
				\end{axis}
		\end{tikzpicture}
	\end{minipage}
	\begin{minipage}{0.45\textwidth}
		\centering
		\begin{tikzpicture}
			\begin{axis}[
				axis lines = none,
				ymin = -3.5, ymax = 3.5,
				xmin = -3.5, xmax = 3.5 
				]
				
				\draw[ samples = 20,	smooth, domain = 0:360, blue, line width = 3pt,
				variable = \t	] plot ( 
				{cos(\t)}, {sin(\t)}
				);
				
				\draw[rounded corners, blue, line width = 3pt] (axis cs: {-2},{-2})--(axis cs: {-2},{2})--(axis cs: {2},{2})--(axis cs: {2},{-2})--cycle;

				\draw (axis cs: {-2},{0})--(axis cs: {cos(180)},{sin(180)});
				
				\draw (axis cs: {-2},{-0.6}) to[out=0,in=200] (axis cs: {cos(200)},{sin(200)});
				\draw (axis cs: {-2},{-1.2}) to[out=0,in=215]  (axis cs: {cos(215)},{sin(215)});
				
				\draw (axis cs: {-1.95},{-1.95})--(axis cs: {cos(225)},{sin(225)});

				\draw (axis cs: {-1.2},{-2}) to[out=90,in=235]  (axis cs: {cos(235)},{sin(235)});
				\draw (axis cs: {-0.6},{-2}) to[out=90,in=250] (axis cs: {cos(250)},{sin(250)});
							
				\draw (axis cs: {0},{-2})--(axis cs: {cos(270)},{sin(270)});
				
				\draw (axis cs: {0.6},{-2}) to[out=90,in=290] (axis cs: {cos(290)},{sin(290)});
				\draw (axis cs: {1.2},{-2}) to[out=90,in=305]  (axis cs: {cos(305)},{sin(305)});
				
				\draw (axis cs: {1.95},{-1.95})--(axis cs: {cos(315)},{sin(315)});
				
				\draw (axis cs: {2},{-1.2}) to[out=180,in=325]  (axis cs: {cos(325)},{sin(325)});
				\draw (axis cs: {2},{-0.6}) to[out=180,in=340] (axis cs: {cos(340)},{sin(340)});
				
				\draw (axis cs: {2},{0})--(axis cs: {cos(0)},{sin(0)});
				
				\draw (axis cs: {2},{0.6}) to[out=180,in=20] (axis cs: {cos(20)},{sin(20)});
				\draw (axis cs: {2},{1.2}) to[out=180,in=35]  (axis cs: {cos(35)},{sin(35)});
				
				\draw (axis cs: {1.95},{1.95})--(axis cs: {cos(45)},{sin(45)});
				
				\draw (axis cs: {1.2},{2}) to[out=270,in=55]  (axis cs: {cos(55)},{sin(55)});
				\draw (axis cs: {0.6},{2}) to[out=270,in=70] (axis cs: {cos(70)},{sin(70)});
				
				\draw (axis cs: {0},{2})--(axis cs: {cos(90)},{sin(90)});
				
				\draw (axis cs: {-0.6},{2}) to[out=270,in=110] (axis cs: {cos(110)},{sin(110)});
				\draw (axis cs: {-1.2},{2}) to[out=270,in=125]  (axis cs: {cos(125)},{sin(125)});
				
				\draw (axis cs: {-1.95},{1.95})--(axis cs: {cos(135)},{sin(135)});
				
				\draw (axis cs: {-2},{1.2}) to[out=0,in=145]  (axis cs: {cos(145)},{sin(145)});
				\draw (axis cs: {-2},{0.6}) to[out=0,in=160] (axis cs: {cos(160)},{sin(160)});
				
			\end{axis}
		\end{tikzpicture}
	\end{minipage}
	\caption{ {\footnotesize  Illustration of foliation induced by $u:[0,T]\times M\to\mathcal{R}$. \textit{Left}: In blue, the part of $M$, corresponding to $M\bigcap\into(\dom f)$, that is injectively transported; In red, the part of $M$, corresponding to $M\bigcap\bd(\dom f)$, where $u$ fails injectivity. \textit{Right}: The case where $\dom f = \R^d$, and $u$ induces a complete foliation.}\label{fig:Foliation}}
	\end{figure}
	
Recall that for every $m\in M\,\bigcap \,\into(\dom f)$ the set $\Gamma = M\cap B(m,\delta)$ is a  $\mathcal{C}^{1,1}$-submanifold of codimension~1 (see Lemma~\ref{lem-mnfd}) and can naturally be endowed with its Hausdorff measure $\mathcal{H}^{d-1}$. Consequently, we can consider the measure~ $\mu := \mathcal{L}_1\times \mathcal{H}^{d-1}$ over $[0,T]\times \Gamma$, where $\mathcal{L}_1$ is the Lebesgue measure over $[0,T]$. We further endow the set $\mathcal{R} = [f\leq 0] \setminus \into[f\leq T]$ with the usual Lebesgue measure $\mathcal{L}_d$ of $\R^d$.	

\begin{theorem}[control of null sets]\label{thm:CurvesAvoidNullMeasureSets} Let $\mathcal{N} \subset \R^d$ be a null measure set, and assume~\ref{hyp:NonemptyIntArgmin}--\ref{hyp:BoundedInternalCurvature} hold. Let $\bar m\in M\cap\into(\dom f)$ such that $f$ is Lipschitz-continuous near $\bar{m}$. Then, there exist $\varepsilon,\delta>0$ and a subset $A\subset \Gamma := B(\bar{m},\delta)\cap M$ of full measure $($i.e.  $\mathcal{H}^{d-1}(A) = \mathcal{H}^{d-1}(\Gamma))$ such that 
\begin{enumerate}
    \item[$\mathrm{(i)}$] $\Gamma \subset \into(\dom f)$ and $u(\Gamma\times [0,\varepsilon])\subset\into(\dom f)$.
    \item[$\mathrm{(ii)}$] $f$ is Lipschitz on $u(\Gamma\times [0,\varepsilon])$.
     \item[$\mathrm{(iii)}$]  for every $m\in A$, one has
		\[
		\mathcal{L}_1(\{ t\in [0,\varepsilon]\ :\ u(t,m)\in \mathcal{N}  \}) = 0.
		\]	
\end{enumerate}
If $\dom f = \R^d$ and $f$ is locally Lipschitz, then $[0,\varepsilon]\times \Gamma$ can be taken to be $[0,T]\times M$.	
	\end{theorem}

\begin{proof}
Let $\varepsilon>0$ such that $[0,\varepsilon]\times\Gamma$ is the neighborhood that appears in the proof of Proposition~\ref{prop:InvertibleSweeping} so that assertions (i), (ii) hold. Without loss of generality, let us assume that $\mathcal{N}\subset \mathcal{O} := u([0,\varepsilon]\times\Gamma)$. Note first that $\mu=\mathcal{L}_1\times \mathcal{H}^{d-1}$ is a Borel measure over $[0,\varepsilon]\times \Gamma$. This yields that $u:[0,\varepsilon]\times \Gamma\to \mathcal{R}$ is measurable, and therefore so is $\ind_{\mathcal{N}}\circ u$. Furthermore, since  $\ind_{\mathcal{N}}\circ u$ is integrable, we can apply Fubini's theorem (see, e.g., \cite[Theorem~1.22]{EvansGariepy2015Measure}) to get that the mapping $\gamma\in\Gamma\mapsto \int_{0}^{\varepsilon}\ind_{\mathcal{N}}(u(t,\gamma)) dt$ is $\mathcal{H}^{d-1}$-measurable and that
		\[
	\mu(u^{-1}(\mathcal{N}))) \equiv \int_{u^{-1}(\mathcal{N})} 1d\mu =\int_{[0,\varepsilon]\times \Gamma} \ind_{\mathcal{N}}(u(t,\gamma))d\mu(t,\gamma) = \int_{\Gamma} \int_{0}^{\varepsilon}\ind_{\mathcal{N}}(u(t,\gamma)) dt\,d\mathcal{H}^{d-1}(\gamma). 
		\]
Since $u$ is Lipschitz-continuous, its Jacobian $Ju(t,\gamma)=|\mathrm{det}(Du(t, \gamma))|$ is well-defined  $\mu$--a.e. in $[0,\varepsilon]\times \Gamma$. Thus, we can apply the co-area formula (see, e.g., \cite[Theorem 3.10]{EvansGariepy2015Measure}) to write
		\[
		\int_{u^{-1}(\mathcal{N})} Ju(t,\gamma)\, d\mu(t,\gamma) = \int_{\mathcal{O}} \underbrace{\mathcal{H}^0\left(u^{-1}(\mathcal{N})\bigcap u^{-1}(x)\right)}_{\in \{0,1 \}}dx = \int_{\mathcal{N}} dx = \mathcal{L}_d(\mathcal{N}) =0 , 
		\]
where the last equality comes from the fact that $u$ is a bijection between $[0,\varepsilon]\times \Gamma$ and $\mathcal{O}$ and consequently, $\mathcal{H}^0\left(u^{-1}(\mathcal{N})\bigcap u^{-1}(x)\right) =\ind_{\mathcal{N}}(x)$, for all $x\in\mathcal{O}$. Finally, since $u$ is bi-Lipschitz, there exists a constant $c>0$ such that $Ju(t,\gamma)\geq c$ for $\mu$-almost every $(t,\gamma)\in [0,\varepsilon]\times \Gamma$. Thus,
		\begin{align*}
			c\int_{\Gamma}\int_{0}^{\varepsilon}\ind_{\mathcal{N}}(u(t,\gamma))\, dt\,\,d\mathcal{H}^{d-1}(\gamma)& = \, \int_{u^{-1}(\mathcal{N})} cd\mu 
			\, \leq  \,	\int_{u^{-1}(\mathcal{N})} Ju(t,\gamma)\mu(t,\gamma)
			= \int_{\mathcal{N}}dx = 0.
		\end{align*}
		Then, the mapping $\gamma\mapsto \int_{0}^{\varepsilon}\ind_{\mathcal{N}}(u(t,\gamma)) dt$ is zero $\mathcal{H}^{d-1}$-almost everywhere in $\Gamma$, and so
		\[
		\mathcal{L}_1(\{ t\in [0,\varepsilon]\ :\ u(t,\gamma)\in \mathcal{N}  \}) = 0, \quad\mbox{ for  $\mathcal{H}^{d-1}$--\textit{ a.e.\,}} \gamma\in \Gamma.
		\]
		The proof is complete.
\end{proof}
\smallskip

Combining Proposition~\ref{prop:InvertibleSweeping} with Theorem~\ref{thm:CurvesAvoidNullMeasureSets}, we will show that, for locally Lipschitz functions, the mapping $u:[0,T]\times M\to\mathcal{R}$ induces steepest descent curves almost everywhere. Indeed, for the special case where $\dom f = \R^d$ the argument goes as follows: let $\mathcal{N}$ be the set of non-differentiability points of $f$. Then by Rademacher theorem $\mathcal{L}_{d}(\mathcal{N})=0$. For every $m\in A$ (the full measure set given by Theorem~\ref{thm:CurvesAvoidNullMeasureSets}) the set
\[
I_m\:= \{ t\in [0,T]\ :\ u(\cdot,m) \text{ is differentiable at }t\text{ and }f\text{ is differentiable at } u(t,m) \}
\]
must be of full measure. Applying chain rule at every point $t\in I_m$  we deduce:
	\[
	-1 = (f\circ u(\cdot,m))^{\prime}(t) = \slope{f}(u(t,m))u^{\prime}(t,m).
	\]
	Thus, $u^{\prime}(t,m) = - \slope{f}(u(t,m))^{-1}$ for almost every $t\in[0,T]$, yielding that $u(\cdot,m)$ is a steepest descent curve. The following theorem deals with the general case. The proof follows the same idea together with a localization argument.
	\begin{theorem}[Existence of steepest descent curves]\label{thm:SteepestDescentAlmostEverywhere} Let $f:\R^d\to\R\cup\{+\infty\}$ be a lower semicontinuous quasiconvex, which is locally Lipschitz on an open set $\mathcal{O}\subset \dom f$. Assume further that $f$ satisfies~\ref{hyp:NonemptyIntArgmin}--\ref{hyp:BoundedInternalCurvature}. Then, for almost every $x\in \mathcal{O}$, the function $f$ admits a steepest descent curve emanating from $x$.
	\end{theorem}
	
	
	\begin{proof}
Without loss of generality, we may assume that $\mathcal{O}$ is convex and $0\in \mathcal{O}$. Let  $\widetilde{\mathcal{N}}$ be the set of all $x\in \mathcal{O}$ for which $f$ does not admit a steepest descent curve emanating from $x$. (Notice that $\widetilde{\mathcal{N}}\cap \argmin f$ is trivially empty.) \smallskip\newline
Let further $\mathcal{N}$ be the set of all $x\in \mathcal{O}$ for which $f$ is not differentiable at~$x$. By Rademacher's theorem $\mathcal{L}_d({\mathcal{N}})=0$. We shall show that the set $\widetilde{\mathcal{N}}$ is also null.
  
To this end, for every $j\in\N$, let us define
\[
B_j = \left(1-\tfrac{1}{j}\right) \overline{\mathcal{O}}.
\]
Clearly, $B_j \subset\mathcal{O}$ and $\mathcal{O} = \bigcup_{j\in\N} B_j$. Now, fix $\alpha \in (\inf_{\mathcal{O}} f,\sup_{\mathcal{O}} f)$ and set
		\begin{align*}
			K_{\alpha,n,j} &= (B_j\cap [f\leq \alpha])\setminus\into([f\leq \inf_{\mathcal{O}} f+\tfrac{1}{n}]); \text{ and}\\
			\widetilde{\mathcal{N}}_{\alpha,n,j} &= \widetilde{\mathcal{N}}\cap K_{\alpha,n,j}
		\end{align*}
Notice that $K_{\alpha,n,j}$ is compact and does not intersect $\argmin f$. Consequently, by  hypothesis~\ref{hyp:BoundedSlope} we deduce that there exists $c>0$ such that  
$\|\nabla f(x) \|>c$ for every $x\in K_{\alpha,n,j}\setminus \mathcal{N}$.\smallskip\newline 
Denote by $h_{\alpha,n,j}:\R^d\to \R$ any Lipschitz extension of $f$ from $K_{\alpha,n,j}$ to $\R^d$. Applying the co-area formula, we deduce
		\[
		\int_{-\infty}^{\infty} \mathcal{H}^{d-1}(\widetilde{\mathcal{N}}_{\alpha,n,j}\cap h_{\alpha,n,j}^{-1}(t))dt = \int_{\widetilde{\mathcal{N}}_{\alpha,n,j}} \|\nabla h_{\alpha,n,j}(x)\| dx = \int_{\widetilde{\mathcal{N}}_{\alpha,n,j}} \|\nabla f(x)\|dx\,\geq c\,\mathcal{L}_d(\widetilde{\mathcal{N}}_{\alpha,n,j}).
		\]
Therefore
\begin{align*}
\mathcal{L}_d(\widetilde{\mathcal{N}}_{\alpha,n,j}) \, \leq \, \frac{1}{c}\int_{-\infty}^{\infty} \mathcal{H}^{d-1}(\widetilde{\mathcal{N}}_{\alpha,n,j}\cap h_{\alpha,n,j}^{-1}(t))\,dt\, = \, \frac{1}{c}\int_{\inf_{\mathcal{O}} f +\tfrac{1}{n}}^{\alpha} \mathcal{H}^{d-1}(\widetilde{\mathcal{N}}_{\alpha,n,j}\cap f^{-1}(t))\,dt.    
\end{align*}

Let us fix $t\in [\inf_{\mathcal{O}} f + 1/n,\alpha]$, set $M=f^{-1}(t)$ and choose $m\in f^{-1}(t)\bigcap \mathcal{O}$. 
There exist $\delta >0$ and $\varepsilon>0$ such that $\Gamma_{m}:=M\cap B(m, \delta) \subset\mathrm{int}\dom f$ and 
$u([0,\varepsilon]\times\Gamma_m)\subset\mathcal{O}$. Let further $A_m\subset \Gamma_m$ be the full measure subset given by Theorem~\ref{thm:CurvesAvoidNullMeasureSets} for $\Gamma=\Gamma_m$ and the null measure set 
$\mathcal{N}$. Then, for every $\gamma\in A_m$ and almost every $t\in [0,\varepsilon]$, $u(\cdot,\gamma)$ is differentiable at $t$ and $f$ is differentiable at $u(t,\gamma)$.
\medskip

Recall that the boundary of the set $S:=[f\leq f(u(t,\gamma))]$ is a $\mathcal{C}^{1,1}$-manifold around $u(t,\gamma)$ for every $t\in[0,\varepsilon]$ (cf. Lemma~\ref{lem-mnfd}). We deduce that $N(S,u(t,\gamma)) = \R_+\{\nabla f(u(t,\gamma))\}$ whenever $f$ is differentiable at $u(t,\gamma)$. Then, for almost all $t\in[0,\varepsilon]$, we have that $\frac{d}{dt}u(y,\gamma) \in \R_+\{\nabla f(u(t,\gamma))\}$. Using this fact and Proposition~\ref{prop:UniqueSolutionSweeping}, we can apply chain rule to deduce that for almost all $t\in [0,\varepsilon]$
		\[
		1=-(f\circ u(\cdot,\gamma))'(t) = \left\|\frac{d}{dt}u(t,\gamma)\right\|\|\nabla f(u(t,\gamma))\| = \left\|\frac{d}{dt}u(t,\gamma)\right\|\slope{f}(u(t,\gamma)).
		\]
We conclude that for every $\gamma\in A_m$, $u(\cdot,\gamma)$ is a steepest descent curve of $f$ emanating from $\gamma$, and so $\widetilde{\mathcal{N}}_{\alpha,n,j}\cap A_m = \emptyset$. Since $M\cap \mathcal{O}=f^{-1}(t)\cap\mathcal{O}$ is $\sigma$-compact, it can be covered by countably many sets $\{\Gamma_{m_k}\ :\ k\in \N\}$, yielding  
		\[
		\mathcal{H}^{d-1}(\widetilde{\mathcal{N}}_{\alpha,n,j}\cap f^{-1}(t)) \leq \mathcal{H}^{d-1}(f^{-1}(t)\setminus \bigcup A_{m_k}) = 0.
		\]
		
		Since the latter conclusion holds for every $t\in [\inf_{\mathcal{O}} f + 1/n,\alpha]$, we deduce that $\mathcal{L}_d(\widetilde{\mathcal{N}}_{\alpha,n,j}) = 0$. Taking $n\to\infty$, $j\to\infty$ and $\alpha\nearrow \sup_{\mathcal{O}} f$, we deduce that
		\[
		\mathcal{L}_d(\widetilde{\mathcal{N}}\setminus (\argmin_{\mathcal{O}} f\cup \argmax_\mathcal{O} f) ) = 0.
		\]
		
Note that~\ref{hyp:BoundedSlope} yields that all level sets of $f$ (except possibly $\argmin f$) must have empty interior. Thus, $\argmax_{\mathcal{O}} f = [f = \sup_{\mathcal{O}} f]\cap \mathcal{O}$ has null measure.
        Moreover, if $\inf_{\mathcal{O}} f >\min f$, then $\argmin_{\mathcal{O}} f$ also has null measure, while if 
$\inf_{\mathcal{O}} f =\min f$, then $\widetilde{\mathcal{N}}\cap \argmin_{\mathcal{O}} f=\emptyset$.  Thus,
        \[
        \mathcal{L}_d(\widetilde{\mathcal{N}}) = \mathcal{L}_d(\widetilde{\mathcal{N}}\setminus (\argmin_{\mathcal{O}} f\cup \argmax_\mathcal{O} f) ) + \mathcal{L}_d(\widetilde{\mathcal{N}}\cap \argmax_{\mathcal{O}} f) + \mathcal{L}_d(\widetilde{\mathcal{N}}\cap \argmin_{\mathcal{O}} f) = 0.
        \]
   The proof is complete.
	\end{proof}

\section{Regularizing locally Lipschitz quasiconvex}\label{sec:main}

In order to apply the results of Section~\ref{sec:Reversible}, we present a regularization scheme based on the max-convolution operator (see, e.g., \cite{SeegerVolle1995}). Given  two functions $f,g:\R^d\to \overline{\R}$  the max-convolution (or sublevel-convolution) of $f$ and $g$, denoted by $f\diamond g$, is defined as
\begin{equation}\label{eq:Def-MaxConvolution}
	(f\diamond g)(x) := \inf_{w\in\R^d} \max\{ f(x-w),g(w) \}.
\end{equation}
Notice that whenever the infimum of \eqref{eq:Def-MaxConvolution} is exact, we have
\begin{equation}\label{eq:SumSublevel-MaxConvolution}
	[f\diamond g\leq \alpha] = [f\leq \alpha] + [g\leq \alpha]\quad\mbox{ and }\quad [f\diamond g< \alpha] = [f< \alpha] + [g< \alpha],
\end{equation}
for every $\alpha\in \R$. \smallskip\newline
In what follows we simply denote by $\Ball\equiv\Ball_d$ the closed unit ball of $\mathbb{R}^d$. Let us assume $\inf f = 0$. Let $\varepsilon>0$ and let us denote by $I_{\varepsilon\Ball}$ the indicator function of~$\varepsilon\Ball$, that is, 
\[
I_{\varepsilon\Ball}(x)=  \begin{cases}  \phantom{+}0, &\text{ if } x \in \varepsilon\Ball \\ 
+\infty, &\text{ if } x \notin \varepsilon\Ball .
\end{cases}
\]
We focus on a particular max-convolution with $g\equiv I_{\varepsilon\Ball}$, namely, we study the function $f_{\varepsilon}= f\diamond I_{\varepsilon\Ball}$ defined by
\begin{equation}\label{eq:Charac-MaxConvWithIndicator}
	f_{\varepsilon}(x) = (f\diamond I_{\varepsilon\Ball})(x) = \inf_{w\in \Ball} f\left (x - \varepsilon w\right ),\qquad \forall x\in \R^d.
\end{equation}
If $\inf f>-\infty$, we can easily adapt the definition of $f_{\varepsilon}$ bysetting $$f_{\varepsilon} = \inf f + (f-\inf f)\diamond I_{\varepsilon\Ball}.$$ In both cases, the formulae of the sublevel sets is preserved: 
$$[f_{\varepsilon}\leq \alpha] = [f\leq \alpha]+\varepsilon\Ball,\quad \text{ for every } \alpha\in\R.$$  

However, if $f$ is not bounded from below, the max-convolution loses that key property. Thus, for the general case, we consider the following definition.
\begin{definition}\label{def:f-epsilon} For a lower semicontinuous function $f:\R^d\to\R\cup\{+\infty\}$ and $\varepsilon>0$ we define the regularized function $f_{\varepsilon}:\R^d\to\R\cup\{+\infty\}$ as follows:	
\[
	f_{\varepsilon}(x) = \inf_{w\in\Ball} f(x-\varepsilon w).
\]
Notice that $f_{\varepsilon}$ the unique function satisfying that $[f_{\varepsilon}\leq \alpha] = [f\leq \alpha]+\varepsilon\Ball$, for every $\alpha\in\R$.
\end{definition}

In this section we will study the class
\begin{equation}\label{eq:ClassQ}
		\mathcal{Q}  = \{ f:\R^d\to\R \mid f \text{ is quasiconvex and locally Lipschitz} \}.
	\end{equation}
Let us first focus our attention on quasiconvex functions satisfying $\min f = 0$. The general case will be treated in forthcoming Theorem~\ref{thm:LocalAlmostEveryDescent}. The next proposition surveys some relevant properties of the above max-convolution that we will use in the subsequent development.

\begin{proposition}\label{prop:PropertiesMaxConv} Assume that $f:\R^d\to\R\cup\{+\infty\}$ is lower semicontinuous, quasiconvex and $\min f =0$. Let $\varepsilon >0$ and consider the max-convolution  $f_{\varepsilon} = f\diamond I_{\varepsilon\Ball}$. Then, the following properties hold:
	\begin{enumerate}[label=(\roman*),ref=\roman*]
		\item\label{propMaxConv-SmoothBoundary} $[f_{\varepsilon}\leq \alpha] = [f\leq \alpha] + \varepsilon\Ball$, for every $\alpha\geq 0$; therefore $\bd [f_{\varepsilon}\leq\alpha]$ is an $\varepsilon$-prox-regular $\mathcal{C}^{1,1}$-submanifold.\\
		\item \label{propMaxConv-SumEps} For $\varepsilon_1,\varepsilon_2>0$ such that $\varepsilon_1+\varepsilon_2=\varepsilon$, one has that
		\[
		f_{\varepsilon} = f_{\varepsilon_1}\diamond I_{\varepsilon_2\Ball} = f_{\varepsilon_2}\diamond I_{\varepsilon_1\Ball}.
		\]
		
		\item\label{propMaxConv-IneqSlope} For each $x\in \R^d$, $f_{\varepsilon}(x) = f(z)$, where $z= \proj(x; [f\leq f_{\varepsilon}(x)])$. Moreover, one has that
		\[
		\slope{f_{\varepsilon}}(x) \geq \slope{f}(z).
		\] 
        \item \label{propMaxConv-Lipschitz} If $f$ is finite-valued and locally Lipschitz-continuous, then $f_{\varepsilon}$ is also locally Lipschitz-con\-ti\-nuous.\\
	\end{enumerate}	
\end{proposition}

\begin{proof}
    Assertion $(i)$ follows from the fact that the infimum in~\eqref{eq:Charac-MaxConvWithIndicator} is exact and by the 
$\mathcal{C}^{1,1}$-smoothness of the distance function $d(\cdot,[f\leq\alpha])$ on the set $\R^n\setminus [f\leq \alpha]$, which is due to the convexity of $[f\leq \alpha]$: indeed, it suffices to notice that $\bd([f_{\varepsilon}\leq \alpha]) = [d(\cdot,[f\leq \varepsilon]) = \varepsilon]$.\smallskip\newline
To establish $(ii)$, we consider the following straightforward computation:
    \begin{align*}
        (f_{\varepsilon_1}\diamond I_{\varepsilon_2\Ball})(x) &= \inf_{w_2\in \Ball} f_{\varepsilon_1}(x-\varepsilon_2w_2)\\
        &= \inf_{w_2\in \Ball} \inf_{w_1\in \Ball} f(x-\varepsilon_1w_1-\varepsilon_2w_2) = \inf_{w\in \Ball} f(x-\varepsilon w) = f_{\varepsilon}(x). 
    \end{align*}
The assertion follows by interchanging the roles of $\varepsilon_1$ and $\varepsilon_2$. \smallskip\newline
Let us now deal with $(iii)$. The first part of this assertion is straightforward. Indeed, since $z\in [f\leq f_{\varepsilon}]$, one has that $f(z)\leq f_{\varepsilon}(x)$. On the other hand, since the infimum in~\eqref{eq:Def-MaxConvolution} is exact, there exists $w\in \varepsilon\Ball$ such that $f_{\varepsilon}(x) = f(x-w)$. Thus, the distance of $x$ to $[f\leq f_{\varepsilon}]$ is smaller than $\varepsilon$. This yields that $\|x-z\|\leq \varepsilon$, and so $f_{\varepsilon}(x)\leq f(z)$.
    
    Now, let $x\in \dom f_{\varepsilon}$ and let $z = \proj(x; [f\leq f_{\varepsilon}(x)])$. Set $w = x-z$ and notice (as mentioned before) that $\|w\|\leq \varepsilon$. Therefore, for every $y\in \dom f$, $f_{\varepsilon}(y+w) \leq f(y)$. Let us also notice that 
    \[
    \{y+w\ :\ y\in [f\leq f(z)]\}\subset [f_{\varepsilon}\leq f_{\varepsilon}(x)].
    \] 
It follows that
    \begin{align*}
        \slope{f}(z) = \limsup_{[f\leq f(z)]\ni y\to z}\frac{f(z) - f(y)}{d(z,y)}
        &= \limsup_{[f\leq f(z)]\ni y\to z}\frac{f_{\varepsilon}(x) - f(y)}{d(x,y+w)}\\
        &\leq \limsup_{[f\leq f(z)]\ni y\to z}\frac{f_{\varepsilon}(x) - f_{\varepsilon}(y+w)}{d(x,y+w)}\\
        &\leq \limsup_{[f_{\varepsilon}\leq f_{\varepsilon}(x)]\ni y\to x}\frac{f_{\varepsilon}(x) - f_{\varepsilon}(y)}{d(x,y)} = \slope{f_{\varepsilon}}(x). 
    \end{align*}
It remains to establish $(iv)$. This follows from a minor adaptation of \cite[Proposition~3.1]{SeegerVolle1995}. Choose $x\in \R^d$ and $\delta>0$, and take $B = \overline{B}(x,\delta+\varepsilon) = \overline{B}(0,\delta)+\varepsilon\Ball$. Since $B$ is compact, the function $f$ is Lipschitz continuous on $B$. Let us denote by $L_B>0$ this Lipschitz constant and take $y,z\in B(x,\delta)$. Let $w\in \varepsilon\Ball$ be such that $f_{\varepsilon}(y) = f(y+w)$. Then, since $y+w,z+w\in B$ we deduce
    \begin{align*}
    f_{\varepsilon}(z) \leq f(z+w) \leq f(y+w) + L_{B}\|y-z\| = f_{\varepsilon}(y) +L_{B}\|y-z\|.
    \end{align*}
Interchanging the roles of $y$ and $z$ in the above development, we obtain
    \[
    |f_{\varepsilon}(z)- f_{\varepsilon}(y)| \leq L_{B}\|y-z\|,
    \]
which yields that $f_{\varepsilon}$ is Lipschitz continuous on $ B(x,\delta)$. \smallskip\newline
The proof is complete.
\end{proof}


Let us now choose $x_0\in \R^d$ such that $\limSlope{f}(x_0)>\ell$ (recall definition in \eqref{eq:limSlope-def}). Pick $\delta>0$ sufficiently small such that 
\begin{equation}\label{ortega}
\slope{f}(y)>\ell, \text{  for every } y\in \bar B(x_0,\delta) \subset \dom f
\end{equation} 
and consider the function
\begin{equation}\label{eq:h}
h := f + I_{\bar B(x_0, \delta)}.
\end{equation}
Notice that since $\min f = 0$, one has that $\min h \geq 0$.
\medskip

\begin{lemma}\label{lemma:SlopeIntersectionAwayFromZero} Assume that \eqref{ortega} holds and $h$ is given by~\eqref{eq:h}. Then the limiting slope $\limSlope{h}$ is strictly positive on $\dom h\setminus\argmin h$.
\end{lemma}


\begin{proof}
Set $C:=\dom h= \bar B(x_0,\delta)$. Choose $z \in \dom h\setminus\argmin h$ set $\alpha = f(z)$ and $\beta \in (\min h, \alpha)$. Note that the set $[\min h <f< \beta]\cap C$ has nonempty interior and consequently, by construction,
	\begin{align*}
		\sup_{y\in [f\leq \beta]\cap C} d(y, \R^d\setminus C) &\geq d(\proj(x_0;[f\leq \beta]), \R^d\setminus C)\\
		&= \delta- d(x_0,[f\leq \beta]) > 0.
	\end{align*}
	Thus, applying \cite[Lemma 1]{Hoffmann1992}, we deduce
\[
d(z, [h\leq \beta]) =\, d(z, [f\leq \beta]\cap C) \, \leq \, \frac{\delta+d(x_0,[f\leq \beta])}{\delta-d(x_0,[f\leq \beta])}d(z,[f\leq \beta]).
\]
Taking the limit $\beta\nearrow \alpha =f(z)$, we obtain
\begin{align*}
		\slope{h}(z) &= \limsup_{\beta\nearrow \alpha}\frac{\alpha - \beta}{d(z,[h\leq \beta])}\geq\,  \limsup_{\beta\nearrow \alpha}\, \frac{\delta - d(x_0,[f\leq \beta])}{\delta+d(x_0,[f\leq \beta])}\frac{\alpha-\beta}{d(z,[f\leq \beta])} \bigskip \\
		&= \, \frac{\delta - d(x_0,[f\leq \alpha])}{\delta+d(x_0,[f\leq \alpha])}\,\slope{f}(z)\,\geq \, \underbrace{\frac{\delta - d(x_0,[f\leq \alpha])}{\delta+d(x_0,[f\leq \alpha])}\ell}_{:=\phi(\alpha)}.
\end{align*}
Since the function $z\mapsto \phi(h(z))$ is continuous and strictly positive on $\dom h\setminus \argmin h$ the assertion follows. 
\end{proof}
\medskip

\begin{lemma}\label{lemma:Regh} For every $\varepsilon>0$, the function $h_{\varepsilon} = h\diamond I_{\varepsilon\Ball}$ is quasiconvex, coercive,
and satisfies~\ref{hyp:NonemptyIntArgmin}--\ref{hyp:BoundedInternalCurvature}.
\end{lemma}

\begin{proof}
Since $f$ is quasiconvex, it is straightforward that $h$ is quasiconvex. Moreover, by construction $\dom h\equiv \bar B(x_0, \delta)$. Thus, Proposition~\ref{prop:PropertiesMaxConv} entails that $h_{\varepsilon}$ is quasiconvex and coercive, where the last property follows from the fact that the domain $\dom h_{\varepsilon}$ coincides with the compact ball $\bar B(x_0, \delta) +\varepsilon\Ball$.
	\medskip
	
Notice further that coercivity and Proposition~\ref{prop:PropertiesMaxConv}~(\ref{propMaxConv-SmoothBoundary}) yield that the function $h_{\varepsilon}$ verifies~\ref{hyp:NonemptyIntArgmin} and~\ref{hyp:BoundedInternalCurvature}, with $r=\varepsilon$, while~\ref{hyp:BoundedSlope} follows by Lemma~\ref{lemma:SlopeIntersectionAwayFromZero} and Proposition~\ref{prop:PropertiesMaxConv}~(\ref{propMaxConv-IneqSlope}). \smallskip\newline The proof is complete.
\end{proof}


We are now ready to establish the main result of this section, which provides steepest descent curves for almost every point of the regularized function $f_{\varepsilon}$ (where $f\in \mathcal{Q}$ is not necessarily assumed to be bounded from below) stemming from non-lower critical points, in the sense of \eqref{eq:NonLowerCriticalCondition} below. Let us set:
\begin{equation}\label{eq:NonLowerCriticalCondition}
\mathcal{U}_{\varepsilon}:=\{ x\in \R^d\, : \, \limSlope{f} (\proj(x; [f\leq f_{\varepsilon}(x)])) > 0\}
\end{equation}

\begin{theorem}\label{thm:LocalAlmostEveryDescent} Let $f\in\mathcal{Q}$ (not necessarily bounded from below) and $\varepsilon>0$. The regularized function $f_{\varepsilon}$ admits steepest descent curves emanating from almost every $x\in\mathcal{U}_{\varepsilon}$. 
In particular, if $f$ verifies~\ref{hyp:BoundedSlope}, then $f_{\varepsilon}$ admits steepest descent curves emanating from almost every $x\in\dom f_{\varepsilon}$. 	
\end{theorem}

\begin{proof}
Let $x\in \mathcal{U}_{\varepsilon}$, and let $z=\proj(x,[f\leq f_{\varepsilon}(x)])$. Then, there exist $\delta,\ell>0$  such that for all $z^{\prime}\in B(z,2\delta)\cap\dom f$, $\slope{f}(z^{\prime})>\ell$. 
\medskip
	
	Take $r=\min_{\bar B(z,2\delta)} f$ and note that $f$ and $g=\max\{f,r\}$ coincide over $B(z,2\delta)$, and so $f_{\varepsilon}$ and $g_{\varepsilon}$ coincide on a neighborhood of $x$. Thus, we can replace $f$ by $g$ and assume, without losing any generality, that $r=\min g = 0$.
	
	Take $h = g + I_{\bar B(z,\delta)}$. Then, by Lemma~\ref{lemma:SlopeIntersectionAwayFromZero}, $h$ verifies~\ref{hyp:BoundedSlope}. By invoking \cite[Example 4.1]{AttouchWets1993} and Lemma~\ref{lemma:StandardHyp}, we deduce that the mapping $y\mapsto \proj(y,[h\leq g_{\varepsilon}(y)])$ is continuous over the set $[g_{\varepsilon}>\min h]$. This yields that there exists $\eta>0$ such that
	\[
	\proj(y,[h\leq g_{\varepsilon}(y)]) \subset z+\frac{\delta}{2}\Ball,\quad\forall y\in B(x,\eta). 
	\]
	In particular,  for all $y\in B(x,\eta)$ one has that $\proj(y,[g\leq g_{\varepsilon}(y)]) = \proj(y,[h\leq g_{\varepsilon}(y)])$ and so,
	\begin{align*}
	g_{\varepsilon}(y) \geq h_{\varepsilon}(y)
	 &= \inf_{w \in y+\varepsilon\Ball} h(y-w) \\
& \geq  \inf_{w \in y+\varepsilon\Ball} g(y-w) = g_{\varepsilon}(y).
	\end{align*}

	Thus, $g_{\varepsilon}$ and $h_{\varepsilon}$ coincide in $B(x,\eta)$. In particular, $h_{\varepsilon}$ is locally Lipschitz on $B(x,\eta)\cap\into(\dom h_{\varepsilon})$. Now, applying Lemma~\ref{lemma:Regh} and Theorem~\ref{thm:SteepestDescentAlmostEverywhere}, we deduce that $h_{\varepsilon}$, and so $g_{\varepsilon}$, admits steepest descent curves emanating from almost every point in $$B(x,\eta)\bigcap\into(\dom h_{\varepsilon}) = B(x,\eta)\bigcap\into(\dom g_{\varepsilon}).$$ The proof is complete.
\end{proof}


\bigskip

\noindent(\textit{Open questions}) The question of characterizing locally Lipschitz functions that admit steepest descent curves on their domains is challenging. This is open even for the class of quasiconvex functions, where a sufficient condition was obtained in \cite{DrusvyatskiyIoffeLewis2015}: namely, this happens whenever the slope mapping $x\mapsto \slope{f}(x)$ is lower semicontinuous, since in this case, every near-steepest descent curve is also a steepest descent curve. Concurrently, it is not known if Theorem~\ref{thm:LocalAlmostEveryDescent} presented hereby is tight, or if the regularized quasiconvex functions $f_{\varepsilon}$ admits steepest descent curves at every point. In addition, it is still unclear if Theorem~\ref{thm:LocalAlmostEveryDescent} holds true for extended-valued quasiconvex functions that are locally Lipschitz on their domains. The main obstruction seems to  be the max-convolution $f_{\varepsilon}$ does not directly inherit the Lipschitz continuity near the boundary of the regularized domain $\dom f_{\varepsilon}$. Last, but not least, we do not dispose a satisfactory characterization for the slope to be lower semicontinuous. In particular, it is not known if the slope of a regularized function $x\mapsto \slope{f_{\varepsilon}}(x)$ is lower semicontinuous or not.

\bigskip

\noindent\rule{5cm}{1pt} \smallskip\newline\noindent\textbf{Acknowledgements.}
 The first author acknowledges support from
the Austrian Science Fund (FWF \textsc{10.55776/P36344}). The second author was
partially funded by the grant \textsc{Fondecyt 11220586} (ANID, Chile) and by the \textsc{BASAL} fund \textsc{FB210005} (Centers of excellence, Chile).

\bibliographystyle{plain}

	
	\vspace{0.8cm}
	\noindent Aris DANIILIDIS
	
	\medskip
	
	\noindent Institut f\"{u}r Stochastik und Wirtschaftsmathematik, VADOR E105-04
	\newline TU Wien, Wiedner Hauptstra{\ss }e 8, A-1040 Wien\medskip
	\newline\noindent E-mail: \texttt{aris.daniilidis@tuwien.ac.at}
	\newline\noindent\texttt{https://www.arisdaniilidis.at/}
	
	\medskip
	
	\noindent Research supported by the Austrian Science Fund (FWF grant \textsc{10.55776/P36344}).\newline\vspace{0.4cm}
	
	\noindent David SALAS
	
	\medskip
	
	\noindent Instituto de Ciencias de la Ingenieria, Universidad de
	O'Higgins\newline Av. Libertador Bernardo O'Higgins 611, Rancagua, Chile
	\smallskip
	
	\noindent E-mail: \texttt{david.salas@uoh.cl} \newline\noindent
	\texttt{http://davidsalasvidela.cl} \medskip
	
	\noindent Research supported by the grants: \smallskip\newline\textsc{CMM
		FB210005 BASAL} funds for centers of excellence (\textsc{ANID}-Chile)\newline%
	\textsc{FONDECYT 11220586} (Chile)
	\end{document}